\documentclass[12pt,a4paper,english]{amsart}
\usepackage[T1]{fontenc}
\usepackage[latin9]{inputenc}
\usepackage{color}
\usepackage{babel}
\usepackage{prettyref}
\usepackage{units}
\usepackage{amsthm}
\usepackage{amsbsy}
\usepackage{amstext}
\usepackage{amssymb}
\usepackage[unicode=true,pdfusetitle,
 bookmarks=true,bookmarksnumbered=false,bookmarksopen=false,
 breaklinks=false,pdfborder={0 0 0},backref=false,colorlinks=true]
 {hyperref}

\makeatletter

\pdfpageheight\paperheight
\pdfpagewidth\paperwidth

  \theoremstyle{plain}
  \newtheorem{lem}{\protect\lemmaname}
  \theoremstyle{plain}
  \newtheorem*{prop*}{\protect\propositionname}
\theoremstyle{plain}
\newtheorem{thm}{\protect\theoremname}
  \theoremstyle{plain}
  \newtheorem*{cor*}{\protect\corollaryname}

\usepackage{hyperref}
\usepackage{mathrsfs}
\usepackage{mathtools}
\usepackage{euscript}
\usepackage{upgreek}

\makeatother

  \providecommand{\corollaryname}{Corollary}
  \providecommand{\lemmaname}{Lemma}
  \providecommand{\propositionname}{Proposition}
\providecommand{\theoremname}{Theorem}

\begin{document}
\global\long\def\dual#1{#1^{{\scriptscriptstyle \vee}}}

\global\long\def\trans#1{^{t}\!#1}

\global\long\def\set#1#2{\left\{  #1\, |\, #2\right\}  }

\global\long\def\map#1#2#3{#1\!:#2\!\rightarrow\!#3}

\global\long\def\aut#1{\mathrm{Aut}\!\left(#1\right)}

\global\long\def\End#1{\mathrm{End}\!\left(#1\right)}

\global\long\def\id#1{\mathbf{id}_{#1}}

\global\long\def\ip#1{\left[#1\right]}

\global\long\def\uhp{\mathbf{H}}

\global\long\def\sm#1#2#3#4{\left(\begin{smallmatrix}#1  &  #2\cr\cr#3  &  #4\end{smallmatrix}\right)}

\global\long\def\cyc#1{\mathbb{Q}\left[\zeta_{#1}\right]}

\global\long\def\ZN#1{\left(\mathbb{Z}/#1\mathbb{Z}\right)^{\times}}

\global\long\def\Mod#1#2#3{#1\equiv#2\, \left(\mathrm{mod}\, \, #3\right)}

\global\long\def\psl{\mathrm{PSL}_{2}\!\left(\mathbb{Z}\right)}

\global\long\def\SL{\mathrm{SL}_{2}\!\left(\mathbb{Z}\right)}

\global\long\def\gl#1#2{\mathrm{GL}_{#1}\!\left(#2\right)}

\global\long\def\tr#1{\mathrm{Tr}#1}

\global\long\def\mr#1#2#3#4{\rho\!\left(\!\begin{array}{cc}
 #1  &  #2\\
#3  &  #4 
\end{array} \!\right)}

\global\long\def\FA{\mathbb{X}}

\global\long\def\FB{\mathbf{\boldsymbol{\Lambda}}}

\global\long\def\FC{\mathcal{M}}

\global\long\def\FD{\boldsymbol{\Gamma}}

\global\long\def\FE#1{\mathfrak{f}_{#1}}

\global\long\def\FF{\mathbb{C}\!\left[J\right]}

\global\long\def\FG{\upepsilon}

\global\long\def\FH{\mathsf{D}}

\global\long\def\sig#1{\upsigma\!\left(#1\right)}

\global\long\def\sof{\mathsf{M}}

\global\long\def\wof#1#2{\mathcal{M}_{#1}\!\left(#2\right)}

\global\long\def\mof#1#2{\mathsf{M}_{#1}\!\left(#2\right)}

\global\long\def\muf#1#2{\mathsf{M}_{#1}^{\circ}\!\left(#2\right)}

\global\long\def\cof#1#2{\mathsf{S}_{#1}\!\left(#2\right)}

\global\long\def\cuf#1#2{\mathsf{S}_{#1}^{\circ}\!\left(#2\right)}

\global\long\def\kan{\varkappa}

\global\long\def\un{\updelta(\tau)}

\global\long\def\ws{\boldsymbol{\varpi}}

\global\long\def\dll#1{\varGamma_{#1}}

\global\long\def\rem#1#2{#1_{#2}}

\global\long\def\detxi#1{\varDelta{}_{#1}}

\global\long\def\ks{\boldsymbol{\uplambda}_{{\scriptscriptstyle +}}}

\global\long\def\kso{\boldsymbol{\uplambda}_{{\scriptscriptstyle -}}}

\global\long\def\gam#1{\upgamma_{#1}}

\global\long\def\Gmof#1#2{\mathfrak{M}_{#1}\!\left(#2\right)}

\global\long\def\FL#1#2#3{\mathfrak{s}_{#3}\!\left(#1,#2\right)}

\global\long\def\hilb#1#2{P_{#1}\!\left(#2\right)}

\title{A trace formula for vector-valued modular forms}

\author{P. Bantay}

\curraddr{Institute for Theoretical Physics, E�tv�s Lor�nd University, Budapest}

\email{bantay@poe.elte.hu}

\thanks{Work supported by grant OTKA 79005.}


\begin{abstract}
We present a formula for vector-valued modular forms, expressing the
value of the Hilbert-polynomial of the module of holomorphic forms
evaluated at specific arguments in terms of traces of representation
matrices, restricting the weight distribution of the free generators.
\end{abstract}
\maketitle

\section{Introduction}

The classical theory of scalar modular forms \cite{apostol2,Knopp1970,koblitz}
has been a major theme of mathematics for the last two hundred years.
Its applications are numerous, ranging from number theory to topology
and mathematical physics, a showpiece being the mathematics involved
in the proof of Fermat's Last Theorem \cite{diamond}. A major result
is the description of the ring of scalar holomorphic forms as a bivariate
polynomial algebra \cite{Lang,Serre}, which allows to determine explicit
bases for the spaces of holomorphic and cusp forms of different weights.

While the need to generalize the theory to vector-valued forms transforming
according to some higher dimensional representation of the modular
group has been recognized long ago, its systematic development has
begun only recently \cite{Knopp2004,Bantay2006,Mason2007,Bantay2008}.
The importance of vector-valued modular forms for mathematics lies,
besides the intrinsic interest of the subject, in the fact that important
classical problems may be reduced to the study of suitable vector-valued
forms, like the theory of Jacobi forms \cite{Eichler1985} or of scalar
modular forms for finite index subgroups \cite{Selberg1965}; from
a modern perspective, trace functions of vertex operator algebras
\cite{FLM1,Kac} satisfying suitable restrictions also provide important
examples of vector-valued modular forms \cite{Zhu1996}. From the
point of view of theoretical physics, vector-valued modular forms
play an important role in string theory \cite{GSW,Polch} and two-dimensional
conformal field theory \cite{DiFrancesco-Mathieu-Senechal}, as the
basic ingredients (chiral blocks) of torus partition functions and
other correlators. 

The above connections justify amply the interest in obtaining a better
understanding of the spaces of vector-valued modular forms. In this
respect, a most interesting question is to find explicit expressions
for the quantities characterizing these spaces, e.g. the dimension
of holomorphic or cusp forms, in terms of the associated representation
of the modular group. Such results are known for some restricted class
of representations, e.g. those having finite image, but their generalization
is not obvious.

A result of Marks and Mason \cite{Marks2009} states that, for a broad
class of representations, the set of all holomorphic forms (of all
possible weights) is a free module over the ring of scalar holomorphic
forms, whose rank equals the dimension of the representation. This
means that, should one know explicitly a free generating set, one
would have complete control over forms for the given representation;
of course, the difficulty lies in obtaining a set of free generators
from the sole knowledge of the representation. More restricted, but
still very useful and non-trivial information is contained in the
weight distribution of the free generators, which can be encoded in
the Hilbert-polynomial of the module of holomorphic forms, and whose
knowledge is enough to determine, in particular, the dimensions of
the spaces of holomorphic forms of different weights. The aim of the
present paper is to show how one can relate the value of the Hilbert-polynomial,
evaluated at specific arguments, to traces of representation operators:
this restricts to a great extent the weight distribution of the free
generators. The basic idea is to investigate weakly-holomorphic forms
alongside holomorphic ones, leading to an explicit expression for
the determinant of the matrix formed from a set of free generators,
and to use the weight-shifting map to compare such determinants.

\section{Scalar modular forms}

A (scalar) modular form of weight $w$ is a complex-valued function
$f\!:\!\uhp\!\rightarrow\!\mathbb{C}$ that is holomorphic everywhere
in the upper half-plane $\uhp\!=\!\set{\tau}{\mathrm{Im}\tau\!>\!0}$,
and transforms according to the rule
\begin{equation}
f\!\left(\frac{a\tau+b}{c\tau+d}\right)=\left(c\tau+d\right)^{w}f\!\left(\tau\right)\:\label{eq:scalartrans}
\end{equation}
for $\sm abcd\!\in\!\FD\!=\!\SL$. Note that the weight $w$ should
be an even integer for non-trivial forms to exist. A form $f\!\left(\tau\right)$
is called weakly holomorphic if it has at worst finite order poles
in the limit $\tau\!\rightarrow\!\mathsf{i}\infty$, i.e. its Laurent-expansion
in terms of the local uniformizing parameter $q\!=\!\exp\!\left(2\pi\mathsf{i}\tau\right)$
has only finitely many terms with negative exponents; it is holomorphic,
respectively a cusp form, if it is bounded (resp. vanishes) as $\tau\!\rightarrow\!\mathsf{i}\infty$,
meaning that its Laurent-expansion contains only non-negative (resp.
positive) powers of $q$. For non-trivial holomorphic (resp. cusp)
forms to exist the weight should be non-negative (resp. positive).
We'll denote by $\FC_{2k}$ the (infinite dimensional) linear space
of weakly holomorphic forms of weight $2k$, and by $\sof_{2k}$ (resp.
$\mathsf{S}_{2k}$) the finite dimensional subspaces of holomorphic
and cusp forms; clearly, we have the inclusions $\mathsf{S}_{2k}<\sof_{2k}<\FC_{2k}$.
Since the product of (weakly) holomorphic (resp. cusp) forms is again
a (weakly) holomorphic (resp. cusp) form, the direct sums $\FC=\bigoplus_{k\in\mathbb{Z}}\FC_{2k}$
and $\mathsf{M}=\bigoplus_{k=0}^{\infty}\sof_{2k}$ (resp. $\mathsf{S}=\bigoplus_{k=1}^{\infty}\mathsf{S}_{2k}$)
are graded rings.

By a well known result \cite{apostol2,Serre}, $\sof_{0}$ consists
of constants, $\sof_{2}$ is empty, while $\sof_{4}$ and $\sof_{6}$,
each having dimension $1$, are spanned by the Eisenstein series
\begin{align}
E_{4}\!\left(q\right)=\, & 1+240\sum\limits _{n=1}^{\infty}\sigma_{3}\!\left(n\right)q^{n}\,\label{eq:eisexpl4}\\
\intertext{and}E_{6}\!\left(q\right)=\, & 1-504\sum\limits _{n=1}^{\infty}\sigma_{5}\!\left(n\right)q^{n}\,,\label{eq:eisexpl6}
\end{align}
where $\sigma_{k}\!\left(n\right)=\sum_{d|n}d^{k}\,$ is the $k^{\mathrm{th}}$
power sum of the divisors of $n$. What is more, any holomorphic form
may be expressed uniquely as a bivariate polynomial in the Eisenstein
series $E_{4}\!\left(q\right)$ and $E_{6}\!\left(q\right)$, in other
words 
\begin{equation}
\mathsf{M}=\bigoplus_{k=0}^{\infty}\sof_{2k}=\mathbb{C}\!\left[E_{4},E_{6}\right]\,\label{eq:modring}
\end{equation}
as graded rings. On the other hand, there are no cusp forms of weight
less than $12$, while $\mathsf{S}_{12}$ is spanned by the discriminant
form 
\begin{equation}
\Delta\!\left(q\right)\,=\!\frac{1}{1728}\left(E_{4}\!\left(q\right)^{3}-E_{6}\!\left(q\right)^{2}\right)\!=\! q\prod_{n=1}^{\infty}\left(1-q^{n}\right)^{{\scriptscriptstyle 24}}\,,\label{eq:deltadef}
\end{equation}
and any cusp form of weight $k\!\geq\!12$ is the product of $\Delta\!\left(q\right)$
with a holomorphic form of weight $k\!-\!12$, i.e. $\mathsf{S}$
is the principal ideal of $\mathsf{M}$ generated by $\Delta\!\left(q\right)$. 

We shall need the following result.
\begin{lem}
\label{lem:division}
\[
\sof_{2n}\subseteq E_{4}^{\rem n3}E_{6}^{\rem n2}\sof_{12\rem n{\infty}}\,
\]
for any non-negative integer $n$, where $\rem nk$ denotes the (non-negative)
remainder of division of $-n$ by the integer $k$, and 
\begin{equation}
\rem n{\infty}\!=\!\frac{n}{6}-\frac{\rem n2}{2}-\frac{\rem n3}{3}\,.\label{eq:ninfdef}
\end{equation}
\end{lem}
\begin{proof}
Since $\sof\!=\!\mathbb{C}\!\left[E_{4},E_{6}\right]$, any $f\!\in\!\sof_{2n}$
may be written as 
\[
f=\sum_{a,b\geq0}\!\mathit{f}\!\left(a,b\right)E_{4}^{a}E_{6}^{b}
\]
for suitable coefficients $f\!\left(a,b\right)\!\in\!\mathbb{C}$
that vanish unless $2a\!+\!3b\!=\! n$, taking into account the relevant
weights. This last condition can be satisfied only if $\Mod a{\rem n3}3$
and $\Mod b{\rem n2}2$. Now, $0\!\leq\!\rem nk\!<\! k$ for $k\!>\!0$,
and because both exponents $a$ and $b$ are non-negative integers,
it does follow that $a\!\geq\!\rem n3$ and $b\!\geq\!\rem n2$, leading
to the conclusion that both $E_{4}^{\rem n3}$ and $E_{6}^{\rem n2}$
divide $f$. Since $f$ has weight $2n$, the assertion follows.
\end{proof}
Let's now turn to weakly holomorphic forms. According to a classic
result \cite{apostol2,Knopp1970,Lang}, the ring $\FC_{0}$ of scalar
weakly-holomorphic forms of weight $0$ is a univariate polynomial
algebra generated by the Hauptmodul 
\begin{equation}
J\!\left(q\right)\!=\!\dfrac{E_{4}\!\left(q\right)^{3}}{\Delta\!\left(q\right)}-744\!=\! q^{-1}\!+\!196884q\!+\cdots\,,\label{eq:Jdef}
\end{equation}
i.e. $\FC_{0}\!=\!\mathbb{C}\!\left[J\right]$. 
\begin{lem}
\label{lem:scalargen}For an integer $n$, the module $\FC_{2n}$
of weakly holomorphic scalar forms of weight $2n$ is generated over
$\FC_{0}$ by the form 
\begin{equation}
\FE n\!\left(q\right)=E_{4}\!\left(q\right)^{\rem n3}E_{6}\!\left(q\right)^{\rem n2}\Delta\!\left(q\right)^{\rem n{\infty}}\,\,,\label{eq:dnfactdef}
\end{equation}
i.e. 
\begin{equation}
\FC_{2n}\!=\!\FE n\FC_{0}\,\,.\label{eq:scalargen}
\end{equation}
\end{lem}
\begin{proof}
Let $g$ denote a weakly holomorphic scalar form of weight $2n$,
and let $k\!\geq\!0$ denote its valence, i.e. the order of its pole
at $q\!=\!0$. Then the product $f\!=\!\Delta^{k}g$ is a holomorphic
scalar form of weight $12k\!+\!2n$, and \prettyref{lem:division}
applies, i.e. there exists is a holomorphic form $F\!\in\!\sof_{12\left(k+\rem n{\infty}\right)}$
such that 
\[
f\!=\! E_{4}^{\rem n3}E_{6}^{\rem n2}F\:,
\]
and thus 
\[
g\!=\!\Delta^{-k}E_{4}^{\rem n3}E_{6}^{\rem n2}F\!=\!\FE n\Delta^{-\rem n{\infty}-k}F\,.
\]
But the product $\Delta^{-\rem n{\infty}-k}F$ is a weight $0$ weakly
holomorphic scalar form, proving the assertion.
\end{proof}
As a consequence 
\[
\FC\!=\!\mathbb{C}\left[J,\FE 1\right]\,,
\]
as a graded ring. Actually, the forms $\FE n\!\left(\tau\right)\!\in\!\FC_{2n}$
are holomorphic for $n\!>\!1$, and satisfy $\FE{n+6}\!\left(q\right)\!=\!\Delta\!\left(q\right)\FE n\!\left(q\right)$. 
\begin{lem}
\begin{equation}
\frac{\FE n\FE m}{\FE{n+m}}\!=\!\left(\!\frac{E_{4}^{3}}{\Delta}\!\right)^{\!\!\FL nm3}\left(\!\frac{E_{6}^{2}}{\Delta}\!\right)^{\!\!\FL nm2}\!=\!\left(J\!+\!744\right)^{\FL nm3}\!\left(J\!-\!984\right)^{\FL nm2}\:,\label{eq:fnmult}
\end{equation}
where 
\begin{equation}
\FL nm{k\!}=\!\begin{cases}
1 & \mathrm{if}\:\ensuremath{\rem nk+\rem mk}\geq k,\\
0 & \mathrm{otherwise}\:.
\end{cases}\label{eq:sabdef}
\end{equation}
\end{lem}
\begin{proof}
From the definition Eq.\prettyref{eq:dnfactdef}, 
\[
\frac{\FE n\FE m}{\FE{n+m}}=E_{4}^{\rem n3+\rem m3-\rem{\left(n+m\right)}3}E_{6}^{\rem n2+\rem m2-\rem{\left(n+m\right)}2}\Delta^{\rem n{\infty}+\rem m{\infty}-\rem{\left(n+m\right)}{\infty}}\:.
\]
But 
\[
\rem n{\infty}\!+\!\rem m{\infty}\!-\!\rem{\left(n\!+\! m\right)}{\infty}\!=\!\frac{\rem{\left(n\!+\! m\right)}2\!-\!\rem n2\!-\!\rem m2}{2}+\frac{\rem{\left(n\!+\! m\right)}3\!-\!\rem n3\!-\!\rem m3}{3}\,,
\]
so the result follows by taking into account that, for $k>0$ 
\begin{equation}
\rem nk+\rem mk-\rem{\left(n+m\right)}k\!=\! k\FL nmk\:.\label{eq:remadd}
\end{equation}

\end{proof}
We conclude this section with the following counting result: 
\begin{lem}
\label{lem:counting}Let $X$ denote a finite subset of $\mathbb{Z}$,
$k$ a positive integer and $0\!<\! p\!<\! k$. Then the number of
elements $x\!\in\! X$ congruent to $p$ modulo $k$ is given by 
\[
\left|\set{x\!\in\! X}{x\!\in\! k\mathbb{Z}\!+\! p}\right|\!=\!\sum_{x\in X}\FL p{-x}k-\sum_{x\in X}\FL{p\!+\!1}{-x}k\:.
\]
\end{lem}
\begin{proof}
First, let's note that $x\!\in\! k\mathbb{Z}\!+\! p$ is equivalent
to $\rem{\left(-x\right)}k\!=\! p$. From Eq.\eqref{eq:sabdef}, 
\[
\sum_{x\in X}\FL n{-x}k\!=\!\left|\set{x\!\in\! X}{\rem nk\!+\!\rem{\left(-x\right)}k\!\geq\! k}\right|
\]
for any $n\!\in\!\mathbb{Z}$; but for $0\!<\! p\!<\! k$ one has
$\rem pk\!=\! k\!-\! p$, hence 
\[
\sum_{x\in X}\FL p{-x}k\!=\!\left|\set{x\!\in\! X}{\rem{\left(-x\right)}k\!\geq\! p}\right|\:.
\]
It follows that 
\[
\sum_{x\in X}\left(\FL p{-x}k-\FL{p\!+\!1}{-x}k\right)\!=\!\left|\set{x\!\in\! X}{\rem{\left(-x\right)}k\!=\! p}\right|\:.
\]

\end{proof}

\section{Vector-valued modular forms\label{sec:general}}

Let $\rho\!:\!\FD\!\rightarrow\!\gl{}V$ denote a representation of
$\FD\!=\!\SL$ on the finite dimensional linear space $V$, and let
$n$ be an integer. A vector-valued modular form of weight $n$ with
multiplier $\rho$ is a map $\FA\!:\!\uhp\!\rightarrow\! V$ that
is holomorphic everywhere in the upper half-plane $\uhp\!=\!\set{\tau}{\mathrm{Im}\tau\!>\!0}$,
and transforms according to the rule
\begin{equation}
\FA\!\left(\frac{a\tau+b}{c\tau+d}\right)=\left(c\tau+d\right)^{n}\mr abcd\FA\!\left(\tau\right)\:\label{eq:modtrans}
\end{equation}
for $\sm abcd\!\in\!\FD$. One recovers the classical notion of scalar
modular forms when $\rho\!=\!\rho_{0}$ is the trivial (identity)
representation.

As in the scalar case, a form is weakly holomorphic if it has at worst
finite order poles in the limit $\tau\!\rightarrow\!\mathsf{i}\infty$,
i.e. its Puisseux-expansion in terms of the local uniformizing parameter
$q\!=\!\exp\!\left(2\pi\mathsf{i}\tau\right)$ contains only finitely
many negative powers of $q$; it is holomorphic, respectively a cusp
form if it is bounded (resp. vanishes) as $\tau\!\rightarrow\!\mathsf{i}\infty$,
meaning that its Puisseux-expansion contains only non-negative (resp.
positive) powers of $q$. We'll denote by $\wof n{\rho}$ the linear
space of weakly holomorphic forms of weight $n$, and by $\mof n{\rho}$
and $\cof n{\rho}$ the subspaces of holomorphic and cusp forms; clearly,
we have the inclusions $\cof n{\rho}\!<\!\mof n{\rho}\!<\!\wof n{\rho}$.
An obvious but important observation is that 
\begin{equation}
\wof n{\rho_{1}\!\oplus\!\rho_{2}}=\wof n{\rho_{1}}\oplus\wof n{\rho_{2}}\:\label{eq:decomp}
\end{equation}
for any two representations $\rho_{1}$ and $\rho_{2}$, and a similar
decomposition holds for the spaces of holomorphic and cusp forms,
allowing to reduce the general theory to the case of indecomposable
representations.

We'll call a representation $\rho\!:\!\FD\!\rightarrow\!\gl{}V$ even
in case $\rho\!\sm{\textrm{-}1}00{\textrm{-}1}\!=\!\id V$, and odd
if $\rho\!\sm{\textrm{-}1}00{\textrm{-}1}\!=\!-\id V$. Any representation
$\rho$ may be decomposed uniquely into a direct sum $\rho\!=\!\rho_{+}\!\oplus\!\rho_{-}$
of even and odd representations, and any indecomposable (in particular,
any irreducible) representation is either even or odd. Combining this
result with Eq.\eqref{eq:decomp}, one gets that it is enough to treat
separately purely even and odd representations, as the general case
can be reduced to these. Note that it follows from Eq.\eqref{eq:modtrans}
that for an even (resp. odd) representation $\rho$ there are no nontrivial
forms of odd (resp. even) weight.

Since the product of a weakly-holomorphic form with a scalar form
$f\!\left(\tau\right)\!\in\!\FC$ is again weakly-holomorphic, it
does follow that the direct sum
\begin{equation}
\wof{}{\rho}=\bigoplus_{n\in\mathbb{Z}}\wof n{\rho}\label{eq:wofdef}
\end{equation}
is a graded module over $\FC$. Similarly, since multiplying a holomorphic
form $\FA\!\left(\tau\right)\!\in\!\mof n{\rho}$ with a scalar holomorphic
form $f\!\left(\tau\right)\!\in\!\sof_{2k}$ results in a new holomorphic
form $f\!\left(\tau\right)\!\FA\!\left(\tau\right)\!\in\!\mof{n+2k}{\rho}$,
and the same is true for cusp forms, the direct sums $\mof{}{\rho}\!=\!\oplus_{n}\mof n{\rho}$
and $\cof{}{\rho}\!=\!\oplus_{n}\cof n{\rho}$ are (graded) modules
over the ring $\mathsf{M}$ of holomorphic scalar modular forms. An
important result of Marks and Mason \cite{Marks2009} states that
$\mof{}{\rho}$ is a free module of rank $d\!=\!\dim\rho$ for a broad
class of representations $\rho$. An interesting question in this
respect is to determine the distribution of the fundamental weights
(i.e. the weights of a set of free generators), which may be answered
by considering the Hilbert-Poincar� series $\Gmof{\rho}z\!=\!\sum_{n}\dim\mof n{\rho}z^{n}$
of this module \cite{Eisenbud}: the number of independent free generators
of weight $k$ equals the coefficient of $z^{k}$ in the Hilbert polynomial
$\hilb{\rho}z\!=\!\left(1\!-\! z^{4}\right)\!\left(1\!-\! z^{6}\right)\Gmof{\rho}z$.

Because the discriminant form $\Delta\!\left(\tau\right)$ does not
vanish on the upper half-plane \cite{apostol2}, its 12th root $\un\!=\! q{}^{\nicefrac{1}{12}}\prod_{n=1}\left(1\!-\! q^{n}\right)^{{\scriptscriptstyle 2}}$
(the square of Dedekind's eta function) is well-defined and holomorphic
on $\uhp$, with an algebraic branch point at the cusp $\tau\!=\!\mathsf{i}\infty$.
Moreover, $\un$ is a weight 1 cusp form with multiplier $\kan$,
where $\kan$ denotes the one dimensional representation of $\SL$
for which%
\footnote{$\kan$ generates the group of linear characters of $\SL$, which
is cyclic of order 12; moreover, $\kan$ is an odd representation,
and tensoring with $\kan$ takes an even representation into an odd
one and \emph{vice versa.} %
}

\begin{equation}
\begin{aligned}\kan\!\sm 0{\textrm{-}1}1{\,0} & \!=-\mathsf{i}\\
\kan\!\sm{\,0}{\,\textrm{-}1}{\,1}{\,\textrm{-}1} & \!=\exp\!\left(\!\dfrac{4\pi\mathsf{i}}{3}\right)\:.
\end{aligned}
\label{eq:canonical-1}
\end{equation}
It does follow that, for any representation $\rho$ and any form $\FA\!\in\!\wof n{\rho}$,
one has $\un^{k}\FA\!\left(\tau\right)\!\in\!\wof{n+k}{\rho\!\otimes\!\kan^{k}}$
for all integers\emph{ }$k\!\in\!\mathbb{Z}$; in other words, one
has a graded bijective map
\begin{align}
\ws_{k}:\wof{}{\rho}\rightarrow & \wof{}{\rho\!\otimes\!\kan^{k}}\label{eq:wsdef}\\
\FA\!\left(\tau\right)\mapsto & \un^{k}\FA\!\left(\tau\right)\,.\nonumber 
\end{align}
The map $\ws_{k}$ relates forms of different weights with a slightly
different multiplier. Note that, since $\un\!\in\!\cof 1{\kan}$ is
a cusp form, multiplication by a positive power of $\un$ takes a
holomorphic form into a cusp form, i.e. $\ws_{k}\!\left(\mof n{\rho}\right)\!<\!\cof{n+k}{\rho\!\otimes\!\kan^{k}}$
for $k\!>\!0$. The bijectivity of the weight-shifting map $\ws_{k}$
allows to reduce to a great extent the study of forms of arbitrary
weights to that of forms of weight $0$ (for a slightly different
representation).

Let's now turn our attention to the properties of $\wof n{\rho}$:
recall that, according to the parity of $\rho$, $n$ has to be even
or odd for this space to be non-trivial. The basic observation is
that the product $f\!\left(\tau\right)\FA\!\left(\tau\right)$ of
a weakly holomorphic form $\FA\!\left(\tau\right)\!\in\!\wof n{\rho}$
with a weakly holomorphic scalar form $f\!\left(\tau\right)\!\in\!\FC_{0}$
of weight $0$ is again a weakly holomorphic form belonging to $\wof n{\rho}$:
in other words, $\wof n{\rho}$ is an $\FC_{0}$-module, that can
be shown to be torsion free. Taking into account the fact that $\FC_{0}$
is the univariate polynomial algebra $\FF$ generated by the Hauptmodul,
this means that actually $\wof n{\rho}$ is a free module \cite{Eisenbud},
whose rank may be shown to equal the dimension $d$ of the representation
$\rho$. This means that there exists forms $\FA_{1},\ldots,\FA_{d}\!\in\!\wof n{\rho}$
that freely generate $\wof n{\rho}$ as an $\FC_{0}$-module, i.e.
any weakly holomorphic form $\FA\!\in\!\wof n{\rho}$ may be decomposed
uniquely into a sum 
\begin{equation}
\FA\!\left(\tau\right)\!=\!\sum_{i=1}^{d}\wp_{i}\!\left(\tau\right)\FA_{i}\!\left(\tau\right)\,,\label{eq:xdecomp}
\end{equation}
where the coefficients $\wp_{1},\ldots,\wp_{d}\!\in\!\FC_{0}$ are
weight $0$ weakly holomorphic scalar forms, i.e. univariate polynomials
in the Hauptmodul $J\!\left(\tau\right)$. 

Given a free generating set $\FA_{1},\ldots,\FA_{d}$ of $\wof n{\rho}$
over $\FC_{0}$, the exterior product $\FA_{1}\wedge\FA_{2}\wedge\ldots\wedge\FA_{d}$
(the determinant of the matrix whose columns are the $\FA_{i}$) is
clearly a weakly holomorphic form of weight $nd$ transforming according
to the one-dimensional determinant representation $\wedge^{d}\rho$
(the $d$-th exterior power of $\rho$): let $\detxi n\!\left(\rho\right)$
denote the quotient of this exterior product by the coefficient of
the lowest power of $q$ in its $q$-expansion. Since the exterior
products of different freely generating sets are proportional, it
follows that $\detxi n\!\left(\rho\right)$ is a well-defined element
of $\wof{nd}{\wedge^{d}\rho}$. 
\begin{lem}
For each integer $n$ 
\begin{equation}
\detxi n\!\left(\rho\right)\!=\!\un^{nd}\detxi 0\!\left(\rho\!\otimes\!\kan^{\mbox{-}n}\right)\:.\label{eq:detform3}
\end{equation}
\end{lem}
\begin{proof}
Since the weight-shifting map $\ws_{-n}$ is bijective, given a free
generating set $\FA_{1},\ldots,\FA_{d}$ of $\wof n{\rho}$, the set
$\ws_{-n}\!\left(\FA_{1}\right),\ldots,\ws_{-n}\!\left(\FA_{d}\right)$
freely generates $\wof 0{\rho\!\otimes\!\kan^{\mbox{-}n}}$, hence
its exterior product is proportional to both $\detxi 0\!\left(\rho\!\otimes\!\kan^{\mbox{-}n}\right)$
and to $\un^{-nd}\detxi n\!\left(\rho\right)$; since the leading
coefficient of both expressions is $1$, the two expressions should
be equal. 
\end{proof}
To conclude this section, we cite the following result from \cite{Bantay2008}.
\begin{prop*}
For an even representation \textup{$\rho\!:\!\FD\!\rightarrow\!\gl{}V$}\textup{\emph{,
one has}}\textup{ 
\begin{equation}
\detxi 0\!\left(\rho\right)=\left(\frac{E_{4}\!\left(q\right)}{\updelta\!\left(q\right)^{4}}\right)^{\beta_{1}+2\beta_{2}}\left(\frac{E_{6}\!\left(q\right)}{\updelta\!\left(q\right)^{6}}\right)^{\alpha}\:,\label{eq:detformula}
\end{equation}
}\textup{\emph{where}} $\alpha$ denotes the multiplicity of $-1$
as an eigenvalue of $\rho\!\sm 0{\textrm{-}1}10$, while $\beta_{1}$
and $\beta_{2}$ denote the multiplicities of $\exp\!\left(\frac{2\pi\mathsf{i}}{3}\right)$
and $\exp\!\left(\frac{4\pi\mathsf{i}}{3}\right)$ as eigenvalues
of $\rho\!\sm 0{\textrm{-}1}1{\textrm{-}1}$.
\end{prop*}
Note that the eigenvalue multiplicities $\alpha$, $\beta_{1}$ and
$\beta_{2}$ can be determined through the relations
\begin{equation}
\begin{aligned}\tr{\,\rho\!\sm 0{\textrm{-}1}10} & =\, d-2\alpha\:,\\
\tr{\,\rho\!\sm 0{\textrm{-}1}1{\textrm{-}1}} & =\, d-\frac{3}{2}\left(\beta_{1}+\beta_{2}\right)+\mathsf{i}\frac{\sqrt{3}}{2}\left(\beta_{1}-\beta_{2}\right)\:.
\end{aligned}
\label{eq:sigtrace}
\end{equation}

\section{The trace formula\global\long\def\rodot{\dot{\rho}}
\global\long\def\wt{k}
}

From now on, we shall assume that $\mof{}{\rho}\!=\!\oplus_{k}\mof k{\rho}$
is a free module of rank $d$ over the ring $\sof\!=\!\mathbb{C}\!\left[E_{4},E_{6}\right]$
of scalar holomorphic forms%
\footnote{According to the result of Marks and Mason mentioned previously \cite{Marks2009},
this holds under rather mild conditions on $\rho$.%
}, and that $\map{\rho}{\FD}{\gl{}V}$ satisfies 
\begin{equation}
\rho\!\sm{\textrm{-}1}00{\textrm{-}1}\!=\!\left(-1\right)^{\FG}\id V\:,\label{eq:epsdef}
\end{equation}
where $\FG\!=\!0$ or $1$ according to whether the representation
$\rho$ is even or odd: note that the representation $\rodot\!=\!\rho\!\otimes\!\kan^{\mbox{-}\FG}$
is always even.
\begin{thm}
\label{thm:weakgens}Let $F_{1},\ldots,F_{d}\!\in\!\mof{}{\rho}$
denote a set of holomorphic forms of respective weights $w_{1},\ldots,w_{d}$,
that generate freely the module $\mof{}{\rho}$. Then the forms $\FE{n-\wt_{1}}F_{1},\ldots,\FE{n-\wt_{d}}F_{d}$
freely generate (over $\FC_{0}$) the module $\wof{2n+\FG}{\rho}$
of weakly holomorphic forms of weight $2n\!+\!\FG$, where 
\[
\wt_{i}=\frac{w_{i}-\FG}{2}=\left[\frac{w_{i}}{2}\right]\:.
\]
\end{thm}
\begin{proof}
Let $\FA\!\in\!\wof{2n+\FG}{\rho}$ denote a weakly holomorphic form
of weight $2n+\FG$. Then, for some integer $z$ large enough, the
product $\Delta\!\left(\tau\right)^{z}\FA\!\left(\tau\right)$ is
holomorphic (the smallest such integer is the valence of $\FA$),
consequently there exist holomorphic scalar forms $x_{1},\ldots,x_{d}\!\in\!\mathsf{M}$
such that 
\[
\Delta\!\left(\tau\right)^{z}\FA\!\left(\tau\right)=\sum_{i=1}^{d}x_{i}F_{i}\,,
\]
with each $x_{i}$ having weight $12z\!+\!2n\!-\!2\wt_{i}$. But each
product $\Delta^{-z}x_{i}$ is a weakly holomorphic scalar form of
weight $2\!\left(n\!-\!\wt_{i}\right)$, hence\global\long\def\hatg{y}
 
\[
\Delta^{-z}x_{i}=\FE{n-\wt_{i}}\hatg_{i}
\]
for some forms $\hatg_{i}\!\in\!\FC_{0}$, according to \prettyref{lem:scalargen}.
As a consequence, 
\[
\FA=\sum_{i=1}^{d}\hatg_{i}\left(\FE{n-\wt_{i}}F_{i}\right)\,,
\]
proving that $\FE{n-\wt_{1}}F_{1},\ldots,\FE{n-\wt_{d}}F_{d}$ generate
$\wof{2n+\FG}{\rho}$; that they are free generators is obvious, since
any relation between them would lead (after multiplication by a suitable
power of the discriminant form) to a relation between the generators
$F_{1},\ldots,F_{d}\!\in\!\mof{}{\rho}$. \end{proof}
\begin{lem}
\label{lem:detformula}
\begin{equation}
\prod_{i=1}^{d}\frac{\FE{n-\wt_{i}}}{\FE{-\wt_{i}}}\!=\!\frac{\detxi{2n+\FG}\!\left(\rho\right)}{\detxi{\FG}\!\left(\rho\right)}\:.\label{eq:detform2}
\end{equation}
\end{lem}
\begin{proof}
By \prettyref{thm:weakgens}, the forms $\FE{n-\wt_{1}}F_{1},\ldots,\FE{n-\wt_{d}}F_{d}$
freely generate $\wof{2n+\FG}{\rho}$, hence their exterior product
is proportional to $\detxi{2n+\FG}\!\left(\rho\right)$: 
\begin{equation}
\detxi{2n+\FG}\!\left(\rho\right)\!=\!\frac{1}{K}\left(\prod_{i=1}^{d}\FE{n-\wt_{i}}\right)\left(F_{1}\wedge F_{2}\wedge\ldots\wedge F_{d}\right)\,,\label{eq:detform1}
\end{equation}
for some nonzero constant of proportionality $K\!\in\!\mathbb{C}$
(whose precise value depends on the choice of the generating set).
The important point is that this constant $K$ equals the coefficient
of the lowest power in the $q$-expansion of the exterior product
$F_{1}\wedge F_{2}\wedge\ldots\wedge F_{d}$, hence it is the same
for all $n$, and the result follows.\end{proof}
\begin{thm}
\label{thm:weightsign} 
\begin{equation}
\begin{aligned}\left|\set i{\wt_{i}\!\in\!2\mathbb{Z}+1}\right|=\, & \alpha\left(\rodot\right)\:,\\
\left|\set i{\wt_{i}\!\in\!3\mathbb{Z}+1}\right|=\, & \beta_{1}\!\left(\rodot\right)\:,\\
\left|\set i{\wt_{i}\!\in\!3\mathbb{Z}+2}\right|=\, & \beta_{2}\!\left(\rodot\right)\:.
\end{aligned}
\label{eq:weightssign}
\end{equation}
\end{thm}
\begin{proof}
According to Eq.\prettyref{eq:detformula}, 
\[
\frac{\detxi 0\!\left(\rodot\!\otimes\!\kan^{\mbox{-}2n}\right)}{\detxi 0\!\left(\rodot\right)}\!=\!\left(\frac{E_{4}\!\left(\tau\right)}{\un^{4}}\right)^{B_{n}}\left(\frac{E_{6}\!\left(\tau\right)}{\un^{6}}\right)^{A_{n}}\:,
\]
where $B_{n}=\beta_{1}\!\left(\rodot\!\otimes\!\kan^{\mbox{-}2n}\right)\!-\!\beta_{1}\!\left(\rodot\right)\!+\!2\beta_{2}\!\left(\rodot\!\otimes\!\kan^{\mbox{-}2n}\right)\!-\!2\beta_{2}\!\left(\rodot\right)$
and $A_{n}=\alpha\!\left(\rodot\!\otimes\!\kan^{\mbox{-}2n}\right)\!-\!\alpha\!\left(\rodot\right)$.
On the other hand, it follows from Eqs.\prettyref{eq:detform2} and
\prettyref{eq:detform3} that 
\begin{multline*}
\frac{\detxi 0\!\left(\rodot\!\otimes\!\kan^{\mbox{-}2n}\right)}{\detxi 0\!\left(\rodot\right)}\!=\!\un^{-2nd}\frac{\detxi{2n+\FG}\!\left(\rho\right)}{\detxi{\FG}\!\left(\rho\right)}\!=\!\un^{-2nd}\prod_{i}\frac{\FE{n-\wt_{i}}\!\left(\tau\right)}{\FE{-\wt_{i}}\!\left(\tau\right)}\\
\!=\!\left(\frac{\FE n\!\left(\tau\right)}{\un^{2n}}\right)^{d}\prod_{i}\frac{\FE{n-\wt_{i}}\!\left(\tau\right)}{\FE n\!\left(\tau\right)\FE{-\wt_{i}}\!\left(\tau\right)}\\
\!=\!\left(\frac{E_{4}\!\left(\tau\right)}{\un^{4}}\right)^{\rem n3d-3\sum_{i}\FL n{\mbox{-}\wt_{i}}3}\left(\frac{E_{6}\!\left(\tau\right)}{\un^{6}}\right)^{\rem n2d-2\sum_{i}\FL n{\mbox{-}\wt_{i}}2}\:,
\end{multline*}
according to Eq.\prettyref{eq:fnmult}. Comparing powers of $E_{4}$
and $E_{6}$ on both sides, one concludes that 
\begin{align*}
3\sum_{i}\FL n{\mbox{-}\wt_{i}}3\!= & \rem n3d-B_{n}\:,\\
2\sum_{i}\FL n{\mbox{-}\wt_{i}}2\!= & \rem n2d-A_{n\:.}
\end{align*}
Because $\kan$ is one dimensional, it is straightforward to compute
$A_{n}$ and $B_{n}$ explicitly, leading to
\[
A_{n}\!=\!\begin{cases}
0 & \textrm{ if }\rem n2\!=\!0\,,\\
d-2\alpha\!\left(\rodot\right) & \textrm{ if }\rem n2\!=\!1\,,
\end{cases}
\]
and

\[
B_{n}\!=\!\begin{cases}
0 & \textrm{ if }\rem n3\!=\!0\,,\\
2d-3\beta_{1}\!\left(\rodot\right)-3\beta_{2}\!\left(\rodot\right) & \textrm{ if }\rem n3\!=\!1\,,\\
d-3\beta_{2}\!\left(\rodot\right) & \textrm{ if }\rem n3\!=\!2\,.
\end{cases}
\]
Finally, combining the above with \prettyref{lem:counting}, one arrives
at 
\begin{align*}
\left|\set i{\wt_{i}\!\in\!3\mathbb{Z}\!+\!1}\right|= & \sum_{i}\!\left(\FL 1{\mbox{-}\wt_{i}}3\!-\!\FL 2{\mbox{-}\wt_{i}}3\right)\!=\!\frac{1}{3}\left(d\!-\! B_{1}\!+\! B_{2}\right)\!=\!\beta_{1}\,,\\
\left|\set i{\wt_{i}\!\in\!3\mathbb{Z}\!+\!2}\right|= & \sum_{i}\!\left(\FL 2{\mbox{-}\wt_{i}}3\!-\!\FL 3{\mbox{-}\wt_{i}}3\right)\!=\!\frac{1}{3}\left(d\!-\! B_{2}\!+\! B_{3}\right)\!=\!\beta_{2}\,,\\
\left|\set i{\wt_{i}\!\in\!2\mathbb{Z}\!+\!1}\right|= & \sum_{i}\!\left(\FL 1{\mbox{-}\wt_{i}}2\!-\!\FL 2{\mbox{-}\wt_{i}}2\right)\!=\!\frac{1}{2}\left(d\!-\! A_{1}\!+\! A_{2}\right)\!=\!\alpha\,.
\end{align*}
\end{proof}
\begin{lem}
\label{lem:gensdet}If $F_{1},\ldots,F_{d}\!\in\!\mof{}{\rho}$ generate
freely the module $\mof{}{\rho}$, then 
\begin{equation}
F_{1}\wedge F_{2}\wedge\ldots\wedge F_{d}\!=\! K\un{}^{\sum_{i}w_{i}}\:.\label{eq:hgensdet}
\end{equation}
\end{lem}
\begin{proof}
It follows from Eqs.\prettyref{eq:detform1} and \prettyref{eq:detformula}
that 
\begin{gather*}
F_{1}\wedge F_{2}\wedge\ldots\wedge F_{d}=K\frac{\detxi{\FG}\!\left(\rho\right)}{\prod_{i}\FE{-\wt_{i}}}=\\
=\frac{K\un^{d\FG}}{\prod_{i}\FE{-\wt_{i}}}\left(\!\frac{E_{4}\!\left(\tau\right)}{\un^{4}}\!\right)^{\!\beta_{1}\left(\rodot\right)+2\beta_{2}\left(\rodot\right)}\left(\!\frac{E_{6}\!\left(\tau\right)}{\un^{6}}\!\right)^{\!\alpha\left(\rodot\right)}\:.
\end{gather*}
On the other hand, 
\[
\prod_{i}\FE{-\wt_{i}}\!=\! E_{4}\!\left(\tau\right)^{\sum_{i}\rem{\left(\mbox{-}\wt_{i}\right)}3}E_{6}\!\left(\tau\right)^{\sum_{i}\rem{\left(\mbox{-}\wt_{i}\right)}2}\Delta\!\left(\tau\right)^{\sum_{i}\rem{\left(\mbox{-}\wt_{i}\right)}{\infty}}\:.
\]
But, according to Eq.\prettyref{eq:weightssign}, 
\begin{align*}
\sum_{i}\rem{\left(\mbox{-}\wt_{i}\right)}3= & \,\beta_{1}\!\left(\rodot\right)+2\beta_{2}\!\left(\rodot\right)\:,\\
\sum_{i}\rem{\left(\mbox{-}\wt_{i}\right)}2= & \,\alpha\!\left(\rodot\right)\:,\\
\sum_{i}\rem{\left(\mbox{-}\wt_{i}\right)}{\infty}= & -\frac{1}{6}\sum_{i}\wt_{i}-\frac{\alpha\!\left(\rodot\right)}{2}-\frac{\beta_{1}\!\left(\rodot\right)+2\beta_{2}\!\left(\rodot\right)}{3}\:,
\end{align*}
leading to 
\[
F_{1}\wedge F_{2}\wedge\ldots\wedge F_{d}\!=\! K\un{}^{d\FG+2\sum_{i}\wt_{i}}\:.
\]
\end{proof}
\begin{cor*}
\label{cor:.weightsum}The sum of the fundamental weights cannot be
negative, $\sum_{i}w_{i}\geq0$.\end{cor*}
\begin{proof}
Should $\sum_{i}w_{i}$ be negative, the exterior product $F_{1}\wedge F_{2}\wedge\ldots\wedge F_{d}$
would not be holomorphic.
\end{proof}
We note that \prettyref{lem:gensdet} and its corollary hold more
generally for any set of forms that generate freely a submodule $\mathsf{m}\!<\!\mof{}{\rho}$
that is $\FH$-stable, i.e. for which $\FH\mathsf{m}\!<\!\mathsf{m}$,
where $\FH$ denotes the covariant derivative.
\begin{thm}
Let $\hilb{\rho}z\!=\!\sum_{i}z^{w_{i}}$ denote the Hilbert-polynomial
of $\sof\!\left(\rho\right)$, and $\zeta\!=\! e^{\frac{2\pi\mathsf{i}}{3}}$
a primitive third root of unity. Then 
\begin{align}
\hilb{\rho}{-\mathsf{i}}= & \,\tr{\,\rho\!\sm 0{\textrm{-}1}10}\label{eq:hilbpoly1}
\end{align}
and
\begin{equation}
\hilb{\rho}{\zeta^{\pm1}}=\,\tr{\,\rho\!\sm 0{\textrm{-}1}1{\textrm{-}1}^{\!\mp1}}\:.\label{eq:hilbpoly2}
\end{equation}
\end{thm}
\begin{proof}
Clearly, since $w_{i}\!=\!2\wt_{i}+\FG$, one has 
\[
\hilb{\rho}{-\mathsf{i}}\!=\!\left(-\mathsf{i}\right)^{\FG}\sum_{i=1}^{d}\left(-1\right)^{\wt_{i}}\:.
\]
But it follows from Eq.\prettyref{eq:weightssign} that 
\begin{gather*}
\sum_{i=1}^{d}\left(-1\right)^{\wt_{i}}\!=\!\left|\set i{\wt_{i}\!\in\!2\mathbb{Z}}\right|-\left|\set i{\wt_{i}\!\in\!2\mathbb{Z}\!+\!1}\right|\!=\! d-2\alpha\!\left(\rodot\right)\\
\!=\!\tr{\,\rodot\!\sm 0{\textrm{-}1}10}\!=\!\kan\!\sm 0{\textrm{-}1}10^{\textrm{-}\FG}\tr{\,\rho\!\sm 0{\textrm{-}1}10}\!=\!\mathsf{i}^{\FG}\tr{\,\rho\!\sm 0{\textrm{-}1}10}\:.
\end{gather*}
Similarly, 
\[
\hilb{\rho}{\zeta^{\pm1}}\!=\!\zeta^{\pm\FG}\sum_{i=1}^{d}\zeta^{\pm2\wt_{i}}\:.
\]
But 
\begin{multline*}
\sum_{i=1}^{d}\zeta^{\pm2\wt_{i}}\!=\!\left|\set i{\wt_{i}\!\in\!3\mathbb{Z}}\right|+\zeta^{\pm2}\left|\set i{\wt_{i}\!\in\!3\mathbb{Z}\!+\!1}\right|+\zeta^{\pm4}\left|\set i{\wt_{i}\!\in\!3\mathbb{Z}\!+\!2}\right|\\
=d-\beta_{1}\!\left(\rodot\right)-\beta_{2}\!\left(\rodot\right)+\beta_{1}\!\left(\rodot\right)e^{{\scriptscriptstyle \mp}\frac{2\pi\mathsf{i}}{3}}+\beta_{2}\!\left(\rodot\right)e^{{\scriptscriptstyle \pm}\frac{2\pi\mathsf{i}}{3}}\!=\!\tr{\,\rodot\!\sm 0{\textrm{-}1}1{\textrm{-}1}^{\!\mp1}}\\
=\!\kan\!\sm 0{\textrm{-}1}1{\textrm{-}1}^{\!\pm\FG}\tr{\,\rho\!\sm 0{\textrm{-}1}1{\textrm{-}1}^{\!\mp1}}=\zeta^{\pm2\FG}\tr{\,\rho\!\sm 0{\textrm{-}1}1{\textrm{-}1}^{\!\mp1}}\:,
\end{multline*}
proving the assertion.
\end{proof}

\section{Summary and outlook}

We have been investigating the distribution of fundamental weights
(i.e. the weights of a freely generating set) of the module $\mof{}{\rho}$
of holomorphic vector-valued modular forms for well behaved representations
$\map{\rho}{\FD}{\gl{}V}$. We established relations between the
weight distribution and properties of the representation through the
consideration of freely generating sets for the modules of weakly
holomorphic forms of diverse weights, leading to the trace formulas
Eqs.\prettyref{eq:hilbpoly1} and \prettyref{eq:hilbpoly2}, expressing
the value of the Hilbert-polynomial $\hilb{\rho}z$ evaluated at special
arguments in terms of traces of representation matrices. These results
could lead to a better understanding of the fundamental weights, and
to effective procedures for determining them from representation theoretic
data. 

Of course, knowing the fundamental weights is far from knowing
explicitly a freely generating set, but this restricted information
could already be very useful.
In this respect, the connection between holomorphic generators and
weakly holomorphic ones of weight $0$, as expressed by \prettyref{thm:weakgens},
seems to be crucial: indeed, for simple examples of low dimension,
such considerations are already enough to compute the $q$-expansions
of a freely generating set. Promoting these\emph{ ad hoc }methods
to some generally valid algorithm would be especially important
for the study and applications of vector-valued modular forms.

\noindent 

\bibliographystyle{plain}

\end{document}